\documentclass{amsart}
\usepackage{graphicx}
\usepackage{amssymb,amscd,amsthm,amsxtra}
\usepackage{latexsym}
\usepackage{epsfig}

\newtheorem{thm}{Theorem}[section]

\newtheorem{lem}[thm]{Lemma}
\newtheorem{prop}[thm]{Proposition}
\theoremstyle{definition}
\newtheorem{defn}[thm]{Definition}
\theoremstyle{remark}
\newtheorem{rem}[thm]{Remark}

\begin{document}

\title[Unbounded Drift Harnack]{Harnack inequality for degenerate and singular elliptic equations with unbounded drift}
\author{Connor Mooney}
\address{Department of Mathematics, Columbia University, New York, NY 10027}
\email{\tt  cmooney@math.columbia.edu}

\begin{abstract}
 We prove a Harnack inequality for functions which, at points of large gradient, are solutions of elliptic equations with unbounded drift.
\end{abstract}
\maketitle
\section{Introduction}
In this paper we consider operators of the form
$$Lu = a^{ij}(x)u_{ij} + b^i(x)u_i$$
on $B_1 \subset \mathbb{R}^n$, $n \geq 2$, where
$$\lambda I \leq a^{ij}(x) \leq \Lambda I$$
are uniformly elliptic, bounded measurable coefficients,
and the drift $b = (b^1,...,b^n)$ is in $L^n(B_1)$ with
$$\|b\|_{L^n(B_1)} = S.$$
We study functions which are solutions to elliptic equations, but only at points where the gradient is large:
\begin{defn}
 Assume $f \in L^n(B_1)$, and $L$ is as above. We say $u \in W^{2,n}(B_1)$ solves
 $$L_{\gamma}u = f$$
 for some $\gamma \geq 0$ if $Lu = f$, but only where $|\nabla u| > \gamma$ (in the Lebesgue point sense).
\end{defn}

Imbert and Silvestre recently studied the case when $|b| \in L^{\infty}$ in \cite{IS}.
The authors prove that such functions satisfy a Harnack inequality and are H\"{o}lder continuous. 
The idea is that the function is already
regular where the gradient is small, and where the gradient is large it solves an equation. The difficulty is that we don't know apriori where the gradient is large.
The key step is an ABP-type estimate which says that if a positive solution is small at some point, then it is small in a set of 
positive measure.
In \cite{IS} the authors obtain this estimate by sliding cusps from below the graph of $u$ until they touch, which ensures that the equation holds at contact points, and 
estimating the measure of these contact points.

Our first main contribution in this paper is a new proof of the measure estimate in \cite{IS} that uses sliding of 
paraboloids from below at all scales and a set decomposition algorithm (see Proposition \ref{MeasureEst}).
While our technique for proving Proposition \ref{MeasureEst} is slightly more involved than sliding cusps, it more directly captures the dichotomy between
contact points at large gradient where the equation holds, and contact points at small gradient where we can rescale to the original situation.

Savin used the idea of applying the equation at contact points with paraboloids in \cite{S} to prove an ABP-type measure estimate. Wang subsequently adapted 
this to the parabolic setting in \cite{W}. It seems hopeful that our technique can also be adapted to prove an analogous measure estimate 
for a class of degenerate parabolic equations. The sliding cusps technique seems difficult to extend to the parabolic setting.

In the remaining parts of this paper we extend the results of \cite{IS} to the situation of unbounded drift.
Heuristically, to get estimates depending on $\|b\|_{L^{n+\epsilon}(B_1)}$
for any $\epsilon > 0$ is no different from doing the usual Krylov-Safonov theory, since under the rescaling $\tilde{u}(x) = u(rx)$ our equation becomes
$$a^{ij}(rx)\tilde{u}_{ij} + rb^i(rx)\tilde{u}_{i} = r^2f,$$
so the new drift term has $L^{n+\epsilon}$ norm $r^{\epsilon/(n+\epsilon)}\|b\|_{L^{n+\epsilon}(B_r)}$, and thus doesn't come into play for $r$ small.

On the other hand, we cannot expect to get estimates depending on $\|b\|_{L^{n-\epsilon}(B_1)}$, where rescaling makes the drift term ``larger.''
Indeed, take the example $\frac{1}{2}|x|^2$, which solves the equation
$$\Delta u - \frac{nx}{|x|^2} \cdot \nabla u = 0.$$
In this simple example $|b| = \frac{n}{|x|}$ which is in $L^{n-\epsilon}$ for any $\epsilon > 0$ but not in $L^n$. 
In this example we violate the strong maximum principle,
a qualitative version of the Harnack inequality.
Thus, having drift in $L^n$ is an interesting critical case. 
Safonov established a Harnack inequality for nondegenerate equations of the form $Lu = f$ in \cite{Saf} 
by working in regions where the drift is small in measure.

In this paper we allow for both unbounded drift and degeneracy of the equation. Our main theorem is a Harnack inequality:
\begin{thm}\label{Main}
 Assume $u$ is a solution to $L_{\gamma}u = 0$ in $B_1$, with $u \geq 0$ and $u(0) = 1$. Then
 $$\sup_{B_{1/2}} u \leq C$$
 where $C$ depends on $\lambda,\Lambda,n,S$ and $\gamma$.
\end{thm}

We also obtain H\"{o}lder regularity of solutions:

\begin{thm}\label{HolderReg}
 If $L_{\gamma}u = 0$ in $B_1$ and $\|u\|_{L^{\infty}(B_1)} \leq 1$ then $u \in C^{\alpha}(B_{1/2})$ and
 $$\|u\|_{C^{\alpha}(B_{1/2})} \leq C(n,\lambda,\Lambda,S,\gamma).$$
\end{thm}

There are several difficulties that arise for equations with unbounded drift that only hold for large gradient. 
The first is that we cannot construct subsolutions to $L$, since we don't know where $|b|$ is large.
We contend with this by sliding standard barriers from below until they touch and obtaining ABP-type measure estimates
(see the proof of Lemma \ref{DoublingLem}). The second is that interior maxima and minima are allowed.
To get around this, we prove a refinement of the ABP maximum principle (see Proposition \ref{RefinedABP}) that controls how far the values of solutions
can be from the boundary data. This allows us to implement ideas of Safonov to prove a measure localization property. Once we have the localization property,
the main theorem follows from standard scaling and covering techniques.

\begin{rem}
 We present our results with right hand side $0$ for clarity and to focus on the role the drift term plays. 
 These results also hold with nonzero right hand side $f$, with constants now depending also on $\|f\|_{L^n(B_1)}$.
\end{rem}

\begin{rem}
The classical Harnack inequality cannot hold for functions which satisfy no equation where the gradient is small. Indeed, $\frac{1}{2}|x|^2$ also
solves an equation with bounded drift where the gradient is large. This is why we require $u(0) = 1$ in the statement
of the Theorem \ref{Main}. The main theorem does not imply the classical Harnack inequality because under multiplication by a large constant,
$\gamma$ is also multiplied by this constant. For example, the ratio of $\sup_{B_{1/2}}u$ to $\inf_{B_{1/2}}u$
is large for $u = \epsilon + \frac{1}{2}|x|^2$, but multiplying by $\frac{1}{\epsilon}$ 
we obtain a function that only solves an equation where the gradient is larger than $\frac{\gamma}{\epsilon}$.
\end{rem}

\begin{rem}
 We treat the case of linear equations with minimal assumptions on the regularity of the coefficients for clarity. This situation
 arises for example by linearizing fully nonlinear equations with the appropriate structure.

 In particular, as noted in \cite{IS}, since we only require that the equation holds for large gradient
 our results also hold for some degenerate and singular elliptic
 equations considered for example in \cite{BD}, \cite{DFQ}, and \cite{D}, only now we also allow unbounded drift.
\end{rem}

The paper is organized as follows. In section $2$ we establish notation and record some simple scaling observations.
In section $3$ we prove the key measure estimate Proposition \ref{MeasureEst}.
In section $4$ we prove a doubling lemma which is standard in the usual Krylov-Safonov theory, but the proof requires modification when an equation only holds where 
the gradient is large and the drift is unbounded. The estimates in sections $3$ and $4$ assume that $S$ is small.
In section $5$ we present a refined form of the ABP maximum principle, which we use in section $6$
along with ideas of Safonov to remove the hypothesis that $S$ is small and prove a measure localization property. 
Theorems \ref{Main} and \ref{HolderReg} then follow in a standard way.


 \section*{Acknowledgement}
  I would like to thank Hector Chang-Lara, Ovidiu Savin and Yu Wang for their encouragement and helpful conversations.
  The author was partially supported by the NSF Graduate Research Fellowship Program under grant number DGE 1144155. 

\section{Preliminaries}
In this section we establish notation and record some simple scaling observations.

\begin{defn}
 We say a paraboloid $P$ has opening $a$ if it can be written 
 $$P(x) = C - \frac{a}{2}|x-x_0|^2$$ 
 for some constant $C$ and $x_0 \in \mathbb{R}^n$.
\end{defn}

\begin{defn}\label{ContactPoints}
 Let $B \subset \mathbb{R}^n$. Slide paraboloids of opening $a$ and vertex in $B$ from below 
 the graph of $u$ until they touch the graph of $u$ by below. We denote by
 $$A_a(B)$$
 the resulting set of contact points by below.
\end{defn}

We remark that in all of the situations below, the sets $A_a(B)$ are in the interior of the domain of definition for $u$. Note that $u \in W^{2,n}$ is in particular
continuous, so its values are unambiguous.

\begin{defn}
 $Q_r(x)$ denotes the cube with side length $r$ centered at $x$. For simplicity we set $Q_r = Q_r(0)$.
\end{defn}

\begin{rem}[{\bf Scaling}]\label{Rescaling}
 Under the rescaling
 $$\tilde{u}(x) = \frac{1}{K}u(rx)$$
 we have
 $$\tilde{L}_{\frac{r\gamma}{K}}\tilde{u} = 0$$
 where $\tilde{L}$ satisfies the same structure conditions as $L$ (recall from the introduction that $S$ remains the same).
 In particular, the equation can only improve if $\frac{r}{K} \leq 1$, so under such rescalings our estimates are invariant.

 Let $\tilde{A}$ denote the contact sets for $\tilde{u}$. Under the above rescaling we also have the relation
 $$\frac{1}{r}A_t(Q_s) = \tilde{A}_{\frac{r^2}{K}t}(Q_{\frac{s}{r}}).$$
\end{rem}

\section{Measure Estimate}
The main proposition of this section says that if a paraboloid touches $u$ by below at an interior point, then $u$ is bounded in a set of positive
measure nearby:

\begin{prop}[{\bf Measure Estimate}]\label{MeasureEst}
 Assume that $L_{\gamma}u \leq 0$ in $B_{4n}$ and that $A_1(0) \cap Q_1$ is nonempty, with
 $$x_0 \in A_1(0) \cap Q_1.$$
 Then there exist $\epsilon_0$ small absolute, $M$ large depending only on $n$ and 
 $\delta, \eta_0$ small depending on $n, \lambda, \Lambda$ such that if $\gamma < \epsilon_0$ and $S < \eta_0$ then
 $$\frac{|\{u \leq u(x_0) + M\} \cap B_{4n}|}{|B_{4n}|} \geq \delta.$$ 
\end{prop}

Our strategy involves sliding paraboloids from below at all scales
and a set decomposition algorithm which we now describe. The starting point is that a paraboloid of opening $1$ centered at $0$ touches $u$ from below in $Q_1$. 
We divide $Q_1$ into dyadic subcubes and slide paraboloids of opening $2$ from below with vertices at the centers of the subcubes.
If a paraboloid touches outside the subcube in which it is centered, then the equation holds at the contact point. 
We exploit this to get a set of positive measure where $u$ is bounded near the contact point, and we don't divide this subcube any further.

On the other hand, if a paraboloid touches in the same subcube in which it is centered, we rescale to the starting point and repeat the procedure.

Iterating this algorithm leads to two possible situations. The first is that after finitely many steps, the measure of the cubes that we stop subdividing is large.
Then by a simple covering argument we show that $u$ is bounded on a set of positive measure nearby contact points of large gradient. 
If this is not the case, then $u$ is bounded at points in arbitrarily small boxes whose total measure is large. Since $u$ is continuous we get that $u$ is bounded
on a set of positive measure in this situation as well.

To start we give a basic ABP-type measure estimate similar to that in \cite{S}. This is the only place that we use the equation in this section.

\begin{lem}\label{ParaboloidABP}
 Assume that $A_a(B) \subset\subset B_1$ for some closed set $B \subset B_R$, $R > 1$, and that $Lu \leq 0$ at all points in $A_a(B)$. Then
 $$|B| \leq C(n,\lambda,\Lambda)\left(|A_a(B)| + R^n S\right).$$
\end{lem}

The idea is that at contact points, we have obvious one-sided control on $D^2u$, and if
an equation holds we get control from both sides on $D^2u$. One then transforms this local information to information in measure via the area formula.

Notice that $A_a(B)$ is closed, hence measurable. 

\begin{proof}
 Assume first that $u$ is $C^2$. 
 Assume that $x \in A_a(y)$ for some $y \in B$. Since $u$ is touched by below with a paraboloid of opening $a$ at $x$ we have the obvious inequality
 $$D^2u(x) \geq -aI.$$
 It follows that
 $$\lambda |(D^2u)^+(x)| - (n-1)\Lambda a \leq a^{ij}(x)u_{ij}(x) \leq |b(x)||\nabla u(x)|.$$
 Since $|x-y| < 2R$ we have $|\nabla u(x)| \leq 2aR$, giving the inequality
 \begin{equation}\label{ParaboloidHessianBound}
  D^2u(x) \leq C(n,\lambda,\Lambda)a(1+|b(x)|R)I.
 \end{equation}
 Writing $y$ in terms of $x$ as
 $$y = x + \frac{1}{a} \nabla u(x)$$
 and differentiating, we obtain
 $$D_xy = I + \frac{1}{a}D^2u.$$
 Inequality (\ref{ParaboloidHessianBound}) gives
 $$0 \leq \det D_xy \leq C(1+R^n|b|^n).$$
 By the area formula, we conclude 
 $$|B| = \int_{A_a(B)} \det D_xy \,dx \leq C\int_{A_a(B)} (1+R^n|b|^n)\,dx$$
 and the estimate follows.

 The case $u \in W^{2,n}(B_1)$ follows from a standard approximation argument (see for example \cite{GT}, Section $9.1$) and the fact that if $u_k \in C^2(B_1)$ 
 converge uniformly to $u$ with contact sets $A_a^k(B)$ then
 $$\cap_{m=1}^{\infty}\cup_{k=m}^{\infty} A_a^k(B) \subset A_a(B).$$
\end{proof}

Next we state a technical but elementary lemma, which we apply at all scales in our algorithm to prove Proposition \ref{MeasureEst}. Up to scaling, the assumption is
that a paraboloid of opening $1$ centered at $0$ touches $u$ by below in $Q_1$. The conclusion is a quantitative
version of the following statements. First, if we slide paraboloids of opening $2$ centered in dyadic subcubes,
they touch $u$ by below nearby. Second, if one of these paraboloids touches outside the subcube in which it is centered, then the equation holds at the contact point,
and we can exploit this to show that $u$ is bounded near the contact point in a set of positive measure.

\begin{lem}\label{LocalABP}
Assume that $L_{\gamma}u \leq 0$ in $B_{4nr}$ and that $\gamma < \frac{1}{4}$.
Assume further that $A_{1/r}(0) \cap Q_r$ is nonempty and $x_0 \in A_{1/r}(0) \cap Q_r$. 
Divide $Q_r$ into dyadic subcubes $\{Q_{r/2}(x_i)\}_{i=1}^{2^n}$. Then the following statements hold:
\begin{enumerate}
 \item $A_{2/r}(x_i) \subset \{u < u(x_0) + nr\} \cap B_{2nr}$, and
 \item There exists $c(n,\lambda,\Lambda)$ such that if $A_{2/r}(x_i) \cap Q_{r/2}(x_i) = \emptyset$ then
 $$\frac{|\{u < u(x_0) + 2nr\} \cap B_{3nr}(x_i)|}{|B_{3nr}|} \geq c - c^{-1}S.$$
\end{enumerate}
\end{lem}
\begin{proof}
 By subtracting a constant we may assume that $u(x_0) = 0$, and under the rescaling 
 $$u(x) \rightarrow \frac{1}{r}u(rx)$$ 
 we may assume that $r = 1$ (see Remark \ref{Rescaling}). Fix $i$ and assume that $P_i(x) := c_i - |x-x_i|^2$ touches $u$ by below.

\vspace{3mm}
{\bf Proof of $(1)$:} Since $P(x_0) \leq 0$ and $|x_i-x_0| < \sqrt{n}$ we have
$$P_i(x) \leq c_i < n,$$
giving
\begin{equation}\label{ContactInequality}
 A_2(x_i) \subset \{u < n\}.
\end{equation}
Next, note that $P_i$ touches $u$ by below in the set $\{P_i(x) > -|x|^2/2\}.$ Using that $c_i \leq n$, this set is contained in
$$\{|x-x_i|^2 -|x|^2/2 < n\}.$$
Using that $|x_i| < \sqrt{n}/2$ and performing an easy computation we obtain that
\begin{equation}\label{LocalizationInequality}
 A_2(x_i) \subset B_{(1 + \sqrt{3})\sqrt{n}} \subset B_{2n}
\end{equation}
for $n \geq 2$. Scaling back proves the claim.

\vspace{3mm}
{\bf Proof of $(2)$:} Assume that $P_i$ touches $u$ by below at $y_i$. Since $y_i \notin Q_{1/2}(x_i)$ and $y_i \in B_{2n}$ we have that
\begin{equation}\label{GradInequality}
 \frac{1}{2} \leq |\nabla u(y_i)| \leq 8n.
\end{equation}
Let $$\tilde{u}(x-y_i) = u(x) - u(y_i) - \nabla u(y_i) \cdot (x-y_i).$$
Then $\tilde{u}(0) = 0$, $\tilde{u} \geq -|x|^2$. Let $\tilde{A}_a$ denote the contact sets for $\tilde{u}$.

 One easily computes that 
\begin{equation}
 \tilde{A}_{4}(B_{1/200}) \subset B_{1/50} \cap \{\tilde{u} < 1\} \cap \{|\nabla \tilde{u}| < 1/4\}. 
\end{equation}

 Combining this with the formula for $u$ and applying Inequalities \ref{ContactInequality}, \ref{LocalizationInequality} and \ref{GradInequality} we obtain that
 $$A_4(B_{1/200}(y_i + \nabla u(y_i)/4)) \subset B_{3n}(x_i) \cap \{u < 2n\} \cap \{|\nabla u| > 1/4\}.$$
 Applying Lemma \ref{ParaboloidABP} to the set $A_4\left(\overline{B_{1/400}(y_i + \nabla u(y_i)/4)}\right)$ we get 
 $$\frac{|\{u < 2n\} \cap B_{3n}(x_i)|}{|B_{3n}|} > c - c^{-1}S$$
 for some small $c(n,\lambda,\Lambda)$. Scaling back proves the claim.
\end{proof}

Lemma \ref{LocalABP} gives a set of positive measure where $u$ is bounded at one scale. 
We next prove a simple covering lemma that ties this information together from scale to scale.

\begin{lem}\label{CoveringLem}
Assume that $\{Q_{r_i}(x_i)\}_{i = 1}^{N}$ is a disjoint collection of cubes in $Q_1$ such that
\begin{equation}\label{MeasureInequality}
 \frac{|\{u < M\} \cap B_{3nr_i}(x_i)|}{|B_{3nr_i}|} \geq 3^n\mu > 0
\end{equation}
for some $M$ and $\mu$. Then
$$|\{u < M\}| \geq \mu |\cup_{i=1}^N Q_{r_i}(x_i)|.$$
\end{lem}
\begin{proof}
 Take a Vitali subcover $\{B^l\}$ of the collection of balls $\{B_{3nr_i}(x_i)\}_{i=1}^N$; that is, a disjoint subcollection such that
 the union of the three-times dilations $\hat{B}^l$ covers $\cup_{i=1}^N B_{3nr_i}(x_i)$. Then
 \begin{align*}
  |\{u < M\}| &\geq \sum_l |\{u < M\} \cap B^l| \\
  &\geq 3^n\mu \sum_l |B^l| \quad (\text{by Inequality } \ref{MeasureInequality}) \\
  &\geq \mu \sum_l |\hat{B}^l| \\
  &\geq \mu \sum_{i=1}^N |Q_{r_i}(x_i)|. 
 \end{align*}
\end{proof}

Finally, we apply Lemma \ref{LocalABP} at all scales to prove Proposition \ref{MeasureEst}.

\begin{proof}[{\bf Proof of Proposition \ref{MeasureEst}}]

In the following, if a cube $Q_{r/2}(x)$ is a dyadic subcube obtained by dividing $Q_r(y)$, we say that $Q_r(y)$ is the predecessor of $Q_{r/2}(x)$.

\vspace{3mm}

 {\bf First Step:}
 Since $A_1(0) \cap Q_1 \neq \emptyset$ we are in the setting of Lemma \ref{LocalABP}. By subtracting a constant assume that
 $$A_1(0) \cap Q_1 \subset \{u < 0\}.$$
 Divide $Q_1$ into dyadic subcubes
 $$\{Q_{1/2}(x_i^1)\}.$$
 If for some $i$ we have 
 $$A_2(x_i^1) \cap Q_{1/2}(x_i^1) = \emptyset,$$ 
 keep the subcube $Q_{1/2}(x_i^1)$. Call the collection of such subcubes $Q^1$. If $Q_{1/2}(x_j) \notin Q^1$, we have by Lemma \ref{LocalABP} that
 $$A_2(x_j^1) \cap Q_{1/2}(x_j^1) \subset \{u < n\}.$$

\vspace{3mm}

 {\bf Inductive Step:} 
 Assume we kept disjoint (possibly empty) collections of disjoint cubes $Q^l = \{Q_{2^{-l}}(x_i^l)\}$ for $l = 1,...,k$ with the following properties:

 \begin{enumerate}
  \item For all $i,l$ we have
  $$A_{2^{l}}(x_i^l) \cap Q_{2^{-l}}(x_i^l) = \emptyset.$$

  \item For arbitrary $i,l$ let $Q_{2^{1-l}}(y)$ be the predecessor for $Q_{2^{-l}}(x_i^l)$. Then
  $$A_{2^{l-1}}(y) \cap Q_{2^{1-l}}(y) \neq \emptyset$$
  and
  $$A_{2^{l-1}}(y) \cap Q_{2^{1-l}}(y) \subset \left\{u < 2n\left(\sum_{m = 0}^{l-1} 2^{-m} - 1\right)\right\}.$$ 

 \item $Q_1 - \cup_{l=1}^{k} Q^l$ is partitioned into cubes 
 $$\{Q_{2^{-k}}(z_i^{k})\}$$
 such that $A_{2^{k}}(z_i^{k}) \cap Q_{2^{-k}}(z_i^{k}) \neq \emptyset$ and
 $$A_{2^{k}}(z_i^{k}) \cap Q_{2^{-k}}(z_i^{k}) \subset \left\{u < 2n\left(\sum_{m=1}^{k} 2^{-m}\right)\right\}.$$
 \end{enumerate}

 We will add a collection $Q^{k+1}$. Divide the cubes $Q_{2^{-k}}(z_i^k)$ into dyadic subcubes $\{Q_{2^{-k-1}}(x_i^{k+1})\}$. If for some $i$ we have 
 $$A_{2^{k+1}}(x_i^{k+1}) \cap Q_{2^{-k-1}}(x_i^{k+1}) = \emptyset,$$
 keep $Q_{2^{-k-1}}(x_i^{k+1})$. Call of the collection of cubes that we keep $Q^{k+1}$. It is obvious by construction that $Q^{k+1}$ is disjoint from
 $\cup_{l=1}^k Q^l$ and that the cubes in $Q^{k+1}$ satisfy
 properties $(1)$ and $(2)$. By Lemma \ref{LocalABP} and property $(3)$ above, in the remaining subcubes $Q_{2^{-k-1}}(x_j^{k+1})$ we have
 $$A_{2^{k+1}}(x_j^{k+1}) \cap Q_{2^{-k-1}}(x_j^{k+1}) \subset \left\{u < 2n\left(\sum_{m=1}^{k+1} 2^{-m}\right)\right\},$$
 completing the inductive step.

\vspace{3mm}

 {\bf Conclusion:}
 Assume first that
 $$|\cup_{l=1}^K Q^l| > \frac{1}{2}$$
 for some $K$. By properties $(1)$ and $(2)$ of the collections $Q^l$ and Lemma \ref{LocalABP}, the cubes $\{Q_{2^{-l}}(x_i^l)\}$ in $\cup_{l=1}^K Q^l$ satisfy the hypotheses of 
 Lemma \ref{CoveringLem} with $M = 4n$, provided we take $\epsilon_0 < \frac{1}{4}$ and $\eta_0$ small depending on $n,\lambda,\Lambda$. (Here $\epsilon_0,\eta_0$ are
 from the statement of Proposition \ref{MeasureEst}). Applying Lemma \ref{CoveringLem} and taking $\delta$ from the statement of Proposition \ref{MeasureEst}
 small, we are done.

 If this never happens, then $u < 2n$ at a point in each of a collection of cubes of arbitrarily small side length whose total measure exceeds $\frac{1}{2}$, and
 we are finished by the continuity of $u$.
\end{proof}

By an easy scaling argument we can state the main result of this section in balls, which will be convenient for later applications:

\begin{prop}\label{BallMeasureEst}
 Assume that $L_{\gamma}u \leq 0$ in $B_2$, $u \geq 0$ and $u \leq 1$ at some point in $B_{1/2}$.
 Then there exist $\epsilon_0$ small absolute, $M$ large depending on $n$ and 
 $\delta, \eta_0$ small depending on $n, \lambda, \Lambda$ such that if $\gamma < \epsilon_0$ and $S < \eta_0$ then
 $$\frac{|\{u \leq M\} \cap B_1|}{|B_1|} \geq \delta.$$ 
\end{prop}

\begin{rem}
 We will remove the restriction that $S$ is small by using a refinement of the ABP maximum principle and techniques due to Safonov (\cite{Saf}) in a later section. 
 (See the proof of Proposition \ref{LocalizationProp}.)
\end{rem}

\section{Doubling Lemma}
In this section we prove a doubling lemma which is standard in the 
Krylov-Safonov theory but requires a new proof when an equation only holds where the gradient is large and the drift is unbounded.
We treat the case when $S$ is small. Again, we will remove this hypothesis in a later section.

\begin{lem}\label{DoublingLem}
  There exist $M, \epsilon,\eta_0$ depending on $n,\lambda,\Lambda$
  such that if $\gamma < \epsilon$, $S < \eta_0$ and  $L_{\gamma}u \leq 0$ in $B_3$,
  $u \geq 0$, and $u \leq 1$ at some point in
  $\overline{B_1}$, then $u \leq M$ at some point in $B_{1/2}$.
\end{lem}

The usual method to prove an estimate like this is to take a barrier of the form $|x|^{-\alpha}$ for $\alpha$ large, cut if off in $B_{1/2}$, and lift
from below. For $|b|$ bounded, we can choose $\alpha$ so that this function is a subsolution and cannot touch except in the cut-off region. However,
in our situation $|b|$ may have large spikes and this function may not be a subsolution. 

To get around this, and to use the equation only at points where the gradient is large, we slide
a barrier of this form from below at many points and use the equation at the contact points.

\begin{proof}
 Assume by way of contradiction that $u(x_0) = 1$ for some point in $\overline{B_1}$ and that $u > M$ everywhere in $B_{1/2}$ for some large $M$ to be determined.
 We first slide barriers from below. Let
 $$\varphi_0(x) = \begin{cases}
                 \frac{|x|^{-\alpha} - 2^{-\alpha}}{4^{\alpha}-2^{-\alpha}} \text{ in } \mathbb{R}^n \backslash B_{1/4}, \\
                 1 \text{ in } B_{1/4}.
                \end{cases}.$$
 In polar coordinates, the Hessian of $|x|^{-\alpha}$ for $\alpha > 0$ is
 $$\alpha|x|^{-\alpha-2}\text{diag}(\alpha + 1,-1,-1,...,-1),$$
 so for $\alpha(n,\lambda,\Lambda)$ large we have $a^{ij}(x)(|x|^{-\alpha})_{ij} > c(n,\lambda,\Lambda) > 0$ in $B_4-\{0\}$.
 Take $C_1$ large enough that 
 $$\varphi(x) := C_1\varphi_0(x) > 2 \text{ in } B_{3/2}$$
 and 
 $$a^{ij}(x)\varphi_{ij} > 1 \text{ in } B_4-B_{1/4}.$$
 Then for $M$ large, if we lift the functions $\varphi(x-y)$ from below $u$ for $y \in B_{1/4}$ they must touch $u$ by below in $B_3$ and outside of
 $\overline{B_{1/4}(y)}$. Furthermore, $c < |\nabla \varphi| < C$ at the contact points by construction. Assuming that $\gamma < c$ the equation
 holds at these points.

 Assume first that $u$ is $C^2$. At a contact point $x$ for a barrier centered at $y$, we thus have
 $$1 \leq a^{ij}(x)u_{ij}(x) \leq |b(x)||\nabla u(x)|$$
 and since $|\nabla u(x)| = |\nabla \varphi(x-y)|$ is bounded above and below by positive universal constants,
 we have that 
 $$|b(x)| > c.$$
 At the contact point $x$ we have the following information:
 \begin{equation}\label{MappingExpression}
  \nabla u(x) = \nabla\varphi(x-y),
 \end{equation}
 \begin{equation}\label{DoublingHessianInequality}
  D^2u(x) \geq D^2\varphi(x-y).
 \end{equation}
 Differentiating expression \ref{MappingExpression} we get
 $$D^2u(x) = D^2\varphi(x-y)(I - D_xy),$$
 which upon rearrangement gives
 $$D_xy = I - D^2u(x)(D^2\varphi)^{-1}(x-y).$$
 Finally, the equation combined with Inequality \ref{DoublingHessianInequality} gives
 $$|D^2u(x)| \leq C(|D^2\varphi(x-y)| + |b(x)||\nabla u(x)|).$$
 Applying this inequality to the expression for $D_xy$ and using that $c < |b(x)|$ we obtain that
 $$|\det D_xy| \leq C|b(x)|^n$$
 on the contact set. Denote the contact set by $E$. Since we lifted the barriers from below for $y \in B_{1/4}$, we obtain
 $$|B_{1/4}| \leq \int_{E} |\det D_xy| \,dx \leq CS,$$
 which would be a contradiction provided that $S < \eta_0$ for some small $\eta_0$ depending on $n,\lambda,\Lambda$.

 The case $u \in W^{2,n}(B_1)$ follows from a standard approximation argument. See for example \cite{GT}, Section $9.1$.
\end{proof}

\section{Refinement of ABP Maximum Principle}

The classical ABP estimate says the following (see \cite{GT}):
\begin{thm}
 Assume that $u$ solves
 $$Lu \leq f$$
 in $B_1$ and $u \geq 0$ on $\partial B_1$. Then
 $$\sup_{B_1} u^- \leq C(n,\lambda,\Lambda,S)\|f\|_{L^n}.$$
\end{thm}
Here $u^-(x) = -\min\{u(x),0\}.$ In particular, if $f \equiv 0$ then $u$ cannot have interior minima. 
This is not true if the equation doesn't hold for small gradient, which we observed in the introduction.
However, we can say something quantitative about how much $u$ can dip below $0$. 
Here we give a refinement of this estimate for functions that only satisfy an equation where the gradient is large.

\begin{prop}\label{RefinedABP}
 There exists $\epsilon$ small depending on $n,\lambda,\Lambda$ and $S$ such that if $\gamma < \epsilon$, $L_{\gamma}u \leq 0$
 in $B_1$, and $u \geq 0$ on $\partial B_1$, then
 $$\sup_{B_1}u^- \leq 1.$$ 
\end{prop}

In the proof we look at the sets where the graph of $u$ has supporting planes with slopes in dyadic annuli. Using the equation at these points
and the observation that $|\nabla u|^n$ and the volumes of the annuli scale the same way, we conclude that $|b|^n$ is large in measure.

\begin{proof}
 Assume not. Then $u$ takes its minimum at some point $x_0 \in B_1$ with $u(x_0) < -1$. Then if we slide any plane with slope in $B_{1/2}$ from
 below $u$, it will touch $u$ by below in $B_1$. If $u$ is $C^2$, then we know $D^2u \geq 0$ at a contact point. From this one-sided bound on $D^2u$
 and the equation (provided it holds for this slope) we obtain that at a contact point,
\begin{equation}\label{PlaneHessianBound}
 \det D^2u \leq C|b|^n|\nabla u|^n.
\end{equation}
 Lift all the planes with slope $p$ in the annulus $R_k = \{2^{-k-1} < |p| \leq 2^{-k}\}$ from below until they touch $u$ by below on some contact set 
 $\Omega_k$. Provided $\epsilon < 2^{-k-1}$, we may apply the area formula and Inequality (\ref{PlaneHessianBound}) to obtain
 $$|R_k| = \int_{\Omega_k} \det D^2u \,dx \leq C\int_{\Omega_k} |b|^n|\nabla u|^n \,dx.$$
 Since $|R_k|$ is $2^{-kn}$ up to a constant depending only on $n$, and furthermore $|\nabla u|^n$ is $2^{-kn}$ up to a constant in $\Omega_k$, we conclude that
 $$\int_{\Omega_k} |b|^n \,dx > c.$$
 Taking $\epsilon = 2^{-k_0}$ we get 
 $$\int_{B_1} |b|^n \,dx > k_0c$$
 which is a contradiction for $k_0$ large.

 Again, the case $u \in W^{2,n}(B_1)$ follows from a standard approximation argument. See for example \cite{GT}, Section $9.1$.
\end{proof}

\begin{rem}
 The proof shows that $\epsilon$ depends upon $S$ like $e^{-CS}$.
\end{rem}

\section{Harnack Inequality}

By combining the measure estimate Proposition \ref{BallMeasureEst} and the doubling Lemma \ref{DoublingLem}, we see that if
$u$ is small at some point, we can localize to a region where $u$ is bounded in a set of positive measure. However,
these estimates used that $S$ is small. We remove this hypothesis in the following proposition with the help of Lemma \ref{RefinedABP}. 

\begin{prop}[{\bf Measure Localization Property}]\label{LocalizationProp}
 There exist small $\epsilon,\delta$ and large $M$ depending on $n,\lambda,\Lambda,S$ 
 such that if $\gamma \leq \epsilon$, $L_{\gamma}u \leq 0$ in $B_2$, $u \geq 0$ in $B_2$ and
 $u \leq 1$ at some point in $\overline{B_1}$, then
 $$\frac{|\{u \leq M\} \cap B_{1/2}|}{|B_{1/2}|} > \delta.$$
\end{prop}

The strategy of the proof is to find some annulus in $B_2-B_1$
where $|b|^n$ is small in average, and connect this annulus to $B_{1/2}$ with a tube in which $|b|^n$ is small in average. We can then apply Lemma
\ref{DoublingLem} in a universal number of small overlapping balls connecting a point in the annulus where $u$ is bounded (the existence of which is guaranteed by
Proposition \ref{RefinedABP})
to a small ball in $B_{1/2}$ where $|b|^n$ is small in average. The conclusion then follows from Proposition \ref{BallMeasureEst}.

\begin{proof}
 First, take a large number $N_1$ of disjoint balls of equal size in $B_{1/2}$. For $N_1$ large depending on $S$, we have that in at least one of these
 balls $B_{r_0}(x_0) = B_0$ that $\|b\|_{L^n(B_0)} < \eta_0$ where $\eta_0$ is the smaller of those from Lemma \ref{DoublingLem} and Proposition \ref{BallMeasureEst}.

 Now take a large number $N_2$ of disjoint tubes of equal radius exiting radially outwards from $B_{0}$ to $\partial B_2$. 
 Then in at least one of these tubes $T_0$, provided $N_2$ is large depending on $S$, 
 we know that $\|b\|_{L^n(T_0)} < \eta_0$. 

 Finally, divide $B_2 \backslash B_1$ into $N_3$ annuli of equal width
 and use the same reasoning to get an annulus $A_0$ such that $\|b\|_{L^n(A_0)} < \eta_0$.

 By Lemma \ref{RefinedABP} we may take $\epsilon$ so small that $u \leq 2$ at the center of some ball of small radius depending on $S$ contained in $A_0$. 
 By connecting a large number $K$ (depending on $S$) of overlapping balls from $A_0$ into $B_0$ through $T_0$ and applying
 Proposition \ref{DoublingLem} (rescaled, see Remark \ref{Rescaling}) in them
 we obtain that $u < 2M_0^{K}$ at a point in $B_{r_0/2}(x_0)$ for some $M_0(n,\lambda,\Lambda)$.

 Having taken $\epsilon$ sufficiently small we may apply Proposition \ref{BallMeasureEst} (rescaled) in $B_0$ to prove the claim.
\end{proof}

We next state the rescaled version of Proposition \ref{LocalizationProp}.
\begin{prop}\label{RescaledMeasureEst}
 There exist $\epsilon,\delta,M$ depending on $n,\lambda,\Lambda,S$ 
 such that for any $r \leq 1$ and $K \geq 1$, if $\gamma \leq \epsilon$, $L_{\gamma}u \leq 0$ in $B_{2r}$, $u \geq 0$ in $B_{2r}$ and
 $u \leq K$ at some point in $\overline{B_r}$, then
 $$\frac{|\{u \leq KM\} \cap B_{r/2}|}{|B_{r/2}|} > \delta.$$
\end{prop}
\begin{proof}
 Apply Proposition \ref{LocalizationProp} to the rescaled function
 $$\tilde{u}(x) = \frac{1}{K}u(rx).$$
 (See Remark \ref{Rescaling}.)
\end{proof}

We note that 
Proposition \ref{RescaledMeasureEst} is the same result as Corollary $4.5$ in \cite{IS}, with constants now depending on $S$ instead of $\|b\|_{L^{\infty}}$.
Theorem \ref{Main} and Theorem \ref{HolderReg} follow from Proposition \ref{RescaledMeasureEst} by standard scaling and covering techniques, and the proofs
are similar to those for the uniformly elliptic case ($\gamma = 0$). The proofs can be taken verbatim from \cite{IS}.

\begin{proof}[{\bf Proofs Theorem \ref{Main} and Theorem \ref{HolderReg}}]
See \cite{IS}, sections $5$,$6$ and $7$.
\end{proof} 



\end{document}